\documentclass[11pt]{amsart}
\usepackage{amsmath, amsfonts, amsbsy, amssymb, upref, setspace}

\begin{document}

\newtheorem{thm}{Theorem}[section]
\newtheorem{lmm}[thm]{Lemma}
\newtheorem{cor}[thm]{Corollary}
\newtheorem{prop}[thm]{Proposition}
\newtheorem{defn}[thm]{Definition}
\newcommand{\argmax}{\operatorname{argmax}}
\newcommand{\argmin}{\operatorname{argmin}}
\newcommand{\avg}[1]{\bigl\langle #1 \bigr\rangle}
\newcommand{\bbb}{\mathbf{B}}
\newcommand{\bbg}{\mathbf{g}}
\newcommand{\bbr}{\mathbf{R}}
\newcommand{\bbw}{\mathbf{W}}
\newcommand{\bbx}{\mathbf{X}}
\newcommand{\bbxa}{\mathbf{X}^*}
\newcommand{\bbxb}{\mathbf{X}^{**}}
\newcommand{\bbxp}{\mathbf{X}^\prime}
\newcommand{\bbxpp}{\mathbf{X}^{\prime\prime}}
\newcommand{\bbxt}{\tilde{\mathbf{X}}}
\newcommand{\bby}{\mathbf{Y}}
\newcommand{\bbz}{\mathbf{Z}}
\newcommand{\bbzt}{\tilde{\mathbf{Z}}}
\newcommand{\bigavg}[1]{\biggl\langle #1 \biggr\rangle}
\newcommand{\bp}{b^\prime}
\newcommand{\bx}{\mathbf{x}}
\newcommand{\by}{\mathbf{y}}
\newcommand{\cc}{\mathbb{C}}
\newcommand{\cov}{\mathrm{Cov}}
\newcommand{\dd}{\mathcal{D}}
\newcommand{\ee}{\mathbb{E}}
\newcommand{\fp}{f^\prime}
\newcommand{\fpp}{f^{\prime\prime}}
\newcommand{\fppp}{f^{\prime\prime\prime}}
\newcommand{\ii}{\mathbb{I}}
\newcommand{\gp}{g^\prime}
\newcommand{\gpp}{g^{\prime\prime}}
\newcommand{\gppp}{g^{\prime\prime\prime}}
\newcommand{\hess}{\operatorname{Hess}}
\newcommand{\ma}{\mathcal{A}}
\newcommand{\mf}{\mathcal{F}}
\newcommand{\mi}{\mathcal{I}}
\newcommand{\ml}{\mathcal{L}}
\newcommand{\cp}{\mathcal{P}}
\newcommand{\mx}{\mathcal{X}}
\newcommand{\mxp}{\mathcal{X}^\prime}
\newcommand{\mxpp}{\mathcal{X}^{\prime\prime}}
\newcommand{\my}{\mathcal{Y}}
\newcommand{\myp}{\mathcal{Y}^\prime}
\newcommand{\mypp}{\mathcal{Y}^{\prime\prime}}
\newcommand{\pp}{\mathbb{P}}
\newcommand{\ppr}{p^\prime}
\newcommand{\pppr}{p^{\prime\prime}}
\newcommand{\ra}{\rightarrow}
\newcommand{\rr}{\mathbb{R}}
\newcommand{\smallavg}[1]{\langle #1 \rangle}
\newcommand{\sss}{\sigma^\prime}
\newcommand{\st}{\sqrt{t}}
\newcommand{\sst}{\sqrt{1-t}}
\newcommand{\tr}{\operatorname{Tr}}
\newcommand{\uu}{\mathcal{U}}
\newcommand{\var}{\mathrm{Var}}
\newcommand{\ve}{\varepsilon}
\newcommand{\vp}{\varphi^\prime}
\newcommand{\vpp}{\varphi^{\prime\prime}}
\newcommand{\ww}{W^\prime}
\newcommand{\xp}{X^\prime}
\newcommand{\xpp}{X^{\prime\prime}}
\newcommand{\xt}{\tilde{X}}
\newcommand{\xx}{\mathcal{X}}
\newcommand{\yp}{Y^\prime}
\newcommand{\ypp}{Y^{\prime\prime}}
\newcommand{\zt}{\tilde{Z}}
\newcommand{\zz}{\mathbb{Z}}
 
\newcommand{\fpar}[2]{\frac{\partial #1}{\partial #2}}
\newcommand{\spar}[2]{\frac{\partial^2 #1}{\partial #2^2}}
\newcommand{\mpar}[3]{\frac{\partial^2 #1}{\partial #2 \partial #3}}
\newcommand{\tpar}[2]{\frac{\partial^3 #1}{\partial #2^3}}

\newcommand{\ts}{\widetilde{S}}
\newcommand{\bw}{\mathbf{w}}
\newcommand{\bbs}{\mathbf{S}}
\newcommand{\bbu}{\mathbf{U}}
\newcommand{\bs}{\mathbf{s}}
\newcommand{\bz}{\mathbf{z}}
\newcommand{\tw}{\widetilde{W}}

\title[A new approach to strong embeddings]{A new approach to strong embeddings}
\author{Sourav Chatterjee}
\address{Courant Institute of Mathematical Sciences, New York University, 251 Mercer Street, New York, NY 10012. {\it E-mail: \tt sourav@cims.nyu.edu}\newline 
}
\thanks{The author's research was partially supported by NSF grant DMS-0707054 and a Sloan Research Fellowship.}

\begin{abstract}
We revisit strong approximation theory from a new perspective, culminating in a proof of the Koml\'os-Major-Tusn\'ady embedding theorem for the simple random walk. The proof is almost entirely based on a series of soft arguments and easy inequalities. 
The new technique, inspired by Stein's method of normal approximation, is applicable to any setting where Stein's method works. In particular, one can hope to take it beyond sums of independent random variables. \end{abstract}
\maketitle

\section{Introduction}\label{intro}
Let $\ve_1,\ve_2,\ldots$ be i.i.d.\ random variables with $\ee (\ve_1) = 0$ and $\ee(\ve_1^2) = 1$. For each $k$, let
\[
S_k = \sum_{i=1}^k \ve_i.
\]
Suppose we want to construct a standard Brownian motion  $(B_t)_{t\ge 0}$ on the same probability space so as to minimize the growth rate of
\begin{equation}\label{max}
\max_{1\le k\le n} |S_k - B_k|.
\end{equation}
Since $S_n$ and $B_n$ both grow like $\sqrt{n}$, one would typically like to have the above quantity growing like $o(\sqrt{n})$, and preferably, as slowly as possible. This is the classical problem of coupling a random walk with a Brownian motion, usually called an `embedding problem' because the most common approach is to start with a Brownian motion and somehow extract the random walk as a process embedded in the Brownian motion.

The study of such embeddings began with the works of Skorohod \cite{skorohod61, skorohod65} and Strassen \cite{strassen67}, who showed that under the condition $\ee(\ve_1^4)< \infty$, it is possible to make \eqref{max} grow like $n^{1/4}(\log n)^{1/2}(\log \log n)^{1/4}$. In fact,  this was shown to be the best possible rate under the finite fourth moment assumption by Kiefer \cite{kiefer69}.  

For a long time, this remained the best available result in spite of numerous efforts by a formidable list of authors to improve on Skorohod's idea.  For a detailed account of these activities, see the comprehensive recent survey of Ob\l\'oj~\cite{obloj04} and the bibliography of the monograph by Cs\"org\H{o} and R\'ev\'esz \cite{csorgorevesz81}. Therefore it came as a great  surprise when Koml\'os, Major, and Tusn\'ady~\cite{kmt75}, almost fifteen years after Skorohod's original work, proved by a completely different argument that one can actually have
\[
\max_{k\le n} |S_k - B_k| = O(\log n)
\]
when $\ve_1$ has a finite moment generating function in a neighborhood of zero.  
Moreover, they showed that this is the best possible result that one can hope for in this situation.
\begin{thm}[Koml\'os-Major-Tusn\'ady \cite{kmt75}]\label{mainthm}
Let  $\ve_1,\ve_2,\ldots$ be i.i.d.\ random variables with $\ee(\ve_1) = 0$, $\ee(\ve_1^2) = 1$, and $\ee \exp\theta|\ve_1| < \infty$ for some $\theta > 0$. For each $k$, let $S_k := \sum_{i=1}^k \ve_i$. Then for any $n$, it is possible to construct a version of $(S_k)_{0\le k\le n}$ and a standard Brownian motion $(B_t)_{0\le t\le n}$  on the same probability space such that for all $x\ge 0$,
\[
\pp\bigl(\max_{k\le n} |S_k - B_k| \ge C\log n + x\bigr) \le K e^{-\lambda x},
\] 
where $C$, $K$, and $\lambda$ do not depend on $n$.
\end{thm}
The paper \cite{kmt75} also contains another very important result, a similar embedding theorem for uniform empirical processes. However, this will not be discussed in this article. See the recent articles by Mason \cite{mason07} and Cs\"org\H{o}~\cite{csorgo07} as well as the book \cite{csorgohorvath93} for more on  the KMT embedding theorem for empirical procceses.

One problem with the proof of Theorem \ref{mainthm}, besides being technically difficult, is that it is very  hard to generalize. Indeed, even the most basic extension to the case of non-identically distributed summands by Sakhanenko \cite{sakhanenko84} is so complex that some researchers are hesitant to use it (see also Shao \cite{shao95}). A nearly optimal multivariate version of the KMT theorem was proved by Einmahl \cite{einmahl89}; the optimal result was obtained by Zaitsev~\cite{zaitsev98} at the end of an extraordinary amount of hard work. More recently, Zaitsev has established multivariate versions of Sakhanenko's theorem \cite{zaitsev00, zaitsev01a, zaitsev01b}. 
For further details and references,  let us refer to the survey article by Zaitsev~\cite{zaitsev02} in the Proceedings of the ICM 2002.

The investigation in this paper is targeted towards a more conceptual understanding of the problem that may allow one to go beyond sums of independent random variables. It begins with the following abstract method of coupling an arbitrary random variable $W$ with a Gaussian random variable $Z$ so that $W-Z$ has exponentially decaying tails at the appropriate scale. (Such a coupling will henceforth be called a {\it strong coupling}, to distinguish it from the `weak' couplings  given by bounds on total variation or Wasserstein metrics.) 
\begin{thm}
\label{coupling}
Suppose $W$ is a random variable with $\ee(W)= 0$ and finite second moment. Let $T$ be another random variable, defined on the same probability space as $W$, such that whenever $\varphi$ is a Lipschitz function and $\vp$ is a derivative of $\varphi$ a.e., we have 
\begin{equation}\label{coeffdef}
\ee(W \varphi(W)) = \ee(\vp(W)T).
\end{equation}
Suppose $|T|$ is almost surely bounded by a constant. Then, given any $\sigma^2 >0$, we can construct $Z \sim N(0,\sigma^2)$ on the same probability space such that 
for any $\theta \in \rr$,
\[
\ee\exp(\theta|W-Z|) \le 2\;\ee\exp\biggl(\frac{2\theta^2(T -\sigma^2)^2}{\sigma^2}\biggr).
\]
\end{thm} 
Let us make a definition here, for the sake of convenience. Whenever $(W,T)$ is a pair of random variables satisfying \eqref{coeffdef}, we will say that $T$ is a {\it Stein coefficient} for $W$. 

The key idea, inspired by Stein's method of normal approximation \cite{stein86}, is that if  $T\simeq \sigma^2$ with high probability, then one can expect that $W$ is approximately Gaussian with mean zero and variance $\sigma^2$. This conclusion is heuristically justified because a random variable $Z$ follows the $N(0,\sigma^2)$ distribution if and only if $\ee(Z\varphi(Z)) = \sigma^2 \ee(\varphi'(Z))$ for all continuously differentiable $\varphi$  such that $\ee|\varphi'(Z)|< \infty$. Stein's method is a process of getting rigorous bounds out of this heuristic.

However, classical Stein's method can only give bounds on quantities like 
\[
\sup_{f\in \mf} |\ee f(W) - \ee f(Z)|,
\]
for various classes of functions $\mf$. This includes, for example, bounds on the total variation distance and the Wasserstein distance, and the Berry-Ess\'een bounds. Theorem \ref{coupling} seems to be of a fundamentally different nature. 

To see how Stein coefficients can be constructed in  a large array of situations, let us consider a few examples.
\medskip

\noindent {\bf Example 1.} Suppose $X$ is a random variable with $\ee(X) = 0$, $\ee(X^2)< \infty$, and following a density $\rho$ that is positive on an interval (bounded or unbounded) and zero outside. Let
\begin{equation}\label{hdef}
h(x) := \frac{\int_x^\infty y \rho(y) dy}{\rho(x)}
\end{equation}
on the support of $\rho$. 
Then, assuming ideal conditions and applying integration by parts, we have $\ee(X\varphi(X)) = \ee(\vp(X) h(X))$ for all Lipschitz $\varphi$.
Thus, $h(X)$ is a Stein coefficient for $X$. The  above  computation is carried out more precisely in Lemma \ref{ipartslemma} in Section \ref{couplingtool}. 
\medskip

\noindent{\bf Example 2.} Suppose $X_1,\ldots, X_n$ are i.i.d.\ copies of the random variable $X$ from the above example, and let $W = \frac{1}{\sqrt{n}} \sum_{i=1}^n X_i$. Then by Example 1,
\begin{align*}
\ee(W \varphi(W)) &= \frac{1}{\sqrt{n}}\sum_{i=1}^n \ee(X_i \varphi(W)) \\
&= \frac{1}{n}\sum_{i=1}^n \ee(h(X_i) \vp(W)) = \ee\biggl(\vp(W) \frac{1}{n}\sum_{i=1}^n h(X_i)\biggr).
\end{align*}
Thus, $\frac{1}{n}\sum_i h(X_i)$ is a Stein coefficient for $W$. Note that this becomes more and more like a constant as $n$ increases, and so we can use Theorem \ref{coupling} to get more and more accurate couplings.
\medskip

\noindent{\bf Example 3.} Suppose $\ve_1,\ldots,\ve_n$ are i.i.d.\ symmetric $\pm1$-valued r.v. Let $S_n = \sum_{i=1}^n \ve_i$.
Let $Y \sim \mathrm{Uniform}[-1,1]$. Let $W_n = S_n + Y$. Let
\[
T_n = n - S_nY + \frac{1-Y^2}{2}.
\]
It will be shown in the proof of Theorem \ref{coupling3} in Section \ref{tusnady} that $T_n$ is a Stein coefficient for $W_n$. (The construction of this $T_n$ is somewhat ad hoc. The author has not yet found a general technique for smoothening of discrete random variables in a way that can automatically generate a Stein coefficient.)
Letting $\sigma^2 = n$, Lemma \ref{coupling} tells us that it is possible to construct $Z_n\sim N(0,n)$ such that 
\[
\ee\exp(\theta|W_n-Z_n|) \le 2\;\ee\exp\biggl(\frac{2\theta^2(T_n-n)^2}{n}\biggr).
\]
Since $T_n = n + O(\sqrt{n})$ and $|W_n-S_n|\le 1$, it is now clear how to use Theorem \ref{coupling} to construct $S_n$ and $Z_n$ on the same probability space such that irrespective of $n$,
\[
\ee\exp(\theta |S_n-Z_n|) \le C
\]
for some fixed constants $\theta$ and $C$. By Markov's inequality, for all $x\ge 0$,
\[
\pp(|S_n-Z_n|\ge x) \le Ce^{-\theta x}.
\]
This is the first step in our proof of the KMT embedding theorem for the simple random walk.
\medskip

\noindent{\bf Example 4.} Suppose $\bbx = (X_1,\ldots,X_n)$ is a vector of  i.i.d.\ standard Gaussian random variables. Let $W = f(\bbx)$, where $f$ is absolutely continuous. Suppose $\ee(W) = 0$.
Let $\bbx' = (X_1',\ldots,X_n')$ be an independent copy of $\bbx$. Let
\[
T = \int_0^1\frac{1}{2\sqrt{t}}\sum_{i=1}^n \fpar{f}{x_i}(\bbx) \fpar{f}{x_i}(\sqrt{1- t}\bbx + \sqrt{t}\bbx') dt. 
\]
Then one can show that $T$ is a Stein coefficient for $W$ (see \cite{chatterjee08}, Lemma~5.3). This has been used to prove CLTs for linear statistics of eigenvalues of random matrices~\cite{chatterjee08}. 
\medskip

\noindent{\bf Example 5.} 
Theorem \ref{coupling} can be used to construct strong  couplings for sums of dependent random variables. An example of such a result is the following.
\begin{thm}\label{auto}
Suppose $X_1,\ldots, X_n, X_{n+1}$ are i.i.d.\ random variables with mean zero, variance $1$, and probability density $\rho$. Suppose $\rho$ is bounded above and below by positive constants on a compact interval, and zero outside. Let $S_n := \sum_{i=1}^n X_i X_{i+1}$.
Then it is possible to construct $S_n$ and a Gaussian random variable $Z_n\sim N(0, n)$ on the same probability space such that for all $x\ge 0$,
\[
\pp(|S_n - Z_n|\ge x) \le e^{-C(\rho) x},
\]
where $C(\rho)$ is a positive constant depending only on the density $\rho$ (and not on $n$). 
\end{thm}
The process $\{S_n\}$, upon proper scaling, is sometimes called the `autocorrelation process' for the sequence~$\{X_n\}$. It may be possible to use the above result to prove a KMT type coupling for autocorrelation processes. The proof of Theorem \ref{auto} is short enough to be presented right here.
\begin{proof}[Proof of Theorem \ref{auto}]
Let $X_0 \equiv 0$. Let $h$ be defined as in \eqref{hdef}. 
Then note that for any $\varphi$, the definition of $h$ and Example 1 show that 
\begin{align*}
\ee(S_n \varphi(S_n)) &= \sum_{i=1}^n \ee(X_i X_{i+1}\varphi(S_n))\\
&= \sum_{i=1}^n \ee(X_{i+1}(X_{i-1}+ X_{i+1})h(X_i)\vp(S_n)).
\end{align*}
This shows that if
\[
D_i:= h(X_i) X_{i+1}(X_{i-1}+X_{i+1}),
\]
then $T_n := \sum_{i=1}^n D_i$ is a Stein coefficient for $S_n$. Now, for any $1\le i\le n$, 
\begin{align*}
\ee(D_i-1 \mid X_1,\ldots, X_{i-1}) = \ee(h(X_i)) \ee(X_{i+1}^2)-1 = 0,
\end{align*}
since $\ee(h(X_i))=\ee(X_i^2)=1$. 
Moreover it is easy to show that by the assumed conditions on $\rho$ that $|D_i|$ is almost surely bounded by a constant depending on $\rho$. Therefore by the Azuma-Hoeffding inequality \cite{hoeffding63, azuma67} for sums of bounded martingale differences, we get that for each $\alpha \in \rr$, 
\[
\ee(e^{\alpha (T_n - n)}) \le e^{C_1(\rho) \alpha^2 n}
\] 
where $C_1(\rho)$ is some constant depending only on $\rho$. Thus if $Z$ is a standard Gaussian random variable, independent of all else, then for any $\alpha \in \rr$ 
\[
\ee(e^{\alpha Z(T_n - n)/\sqrt{n}}) \le \ee(e^{C_1(\rho)Z^2 \alpha^2}). 
\]
Therefore choosing $\alpha = C_2(\rho)$ small enough, one gets
\[
\ee(e^{C_2(\rho)Z(T_n - n)/\sqrt{n}}) \le 2.
\]
On the other hand, first conditioning on $T_n$ we get
\[
\ee(e^{C_2(\rho)Z(T_n - n)/\sqrt{n}}) = \ee(e^{C_2(\rho)^2(T_n - n)^2/2n}).
\]
By Theorem \ref{coupling}, this completes the proof. 
\end{proof}
\medskip 

\noindent {\bf Sketch of the proof of Theorem \ref{coupling}.} (Full details are given in Section~\ref{couplingtool}.)
First, let $h(W) := \ee(T|W)$. 
Then $h(W)$ is again a Stein coefficient for $W$. Moreover, one can show that the function $h$ is non-negative a.e.\ on the support of $W$. It is not difficult to verify that to prove Theorem \ref{coupling} it suffices to construct a coupling such that for all $\theta$,
\[
\ee\exp(\theta|W-Z|) \le 2\; \ee \exp\bigl(2\theta^2 \bigl(\sqrt{h(W)} - \sigma\bigr)^2\bigr).
\]
Fix a function $r:\rr^2 \ra \rr$. For $f\in C^2(\rr^2)$, let
\[
\ml f(x,y) := h(x) \spar{f}{x} + 2r(x,y) \mpar{f}{x}{y} + \sigma^2\spar{f}{y} - x\fpar{f}{x}-y \fpar{f}{y}.
\]
Suppose there exists a probability measure $\mu$ on $\rr^2$ such that for all~$f$,
\begin{equation}\label{inv}
\int_{\rr^2} \ml f \; d\mu = 0.
\end{equation}
The main idea is as follows: {\it every} choice of $r$ that admits a $\mu$ satisfying~ \eqref{inv} gives  a coupling of $W$ and $Z$.
Indeed, suppose $\mu$ is as above and $(X,Y)$ is a random vector with law $\mu$. Take any $\Phi\in C^2(\rr)$, and let $\varphi = \Phi'$. Putting $f(x,y) = \Phi(x)$ in \eqref{inv} gives
\[
\ee(h(X) \vp(X)) = \ee( X\varphi(X)).
\]
Since this holds for all $\varphi$ (which is a property that characterizes $W$) it is possible to argue that $X$ must have the same law as $W$.
Similarly, putting $f(x,y) = \Phi(y)$, we get
$\ee(Y\varphi(Y))= \sigma^2 \ee(\vp(Y))$, and 
thus, $Y\sim N(0,\sigma^2)$. Note that this argument did not depend on the choice of $r$ at all, except through the assumption that there exists a $\mu$ satisfying \eqref{inv}.

Now the question is, for what choices of $r$ does there exist a $\mu$ satisfying~\eqref{inv}? In Lemma \ref{schauder} it is proved that this is possible whenever the matrix 
\[
\left(
\begin{array}{cc}
h(x) & r(x,y)\\
r(x,y) & \sigma^2
\end{array}
\right)
\]
is positive semidefinite for all $(x,y)$, plus some extra conditions. Note that this is the same as saying that the operator $\ml$ is elliptic. 

Intuitively, the `best' coupling of $W$ and $Z$ is obtained when the choice of $r(x,y)$ is such that the matrix displayed above is the `most singular'. This choice is given by the geometric mean
\[
r(x,y) = \sigma \sqrt{h(x)}.
\] 
With this choice of $r$ and $f(x,y) = \frac{1}{2k}(x-y)^{2k}$ (where $k$ is an arbitrary positive integer), a small computation gives 
\[
\ml f(x,y) = (2k-1)(x-y)^{2k-2}(\sqrt{h(x)} - \sigma)^2 - (x-y)^{2k}.
\]
Since \eqref{inv} holds for this $f$, we get
\begin{align*}
\ee(X-Y)^{2k} &= (2k-1)\ee((X-Y)^{2k-2}(\sqrt{h(X)} - \sigma)^2)\\
&\le (2k-1) (\ee(X-Y)^{2k})^{(k-1)/k} (\ee(\sqrt{h(X)} - \sigma)^{2k})^{1/k}.
\end{align*}
This gives
\[
\ee(X-Y)^{2k} \le (2k-1)^k \ee(\sqrt{h(X)} - \sigma)^{2k}.
\]
It is now easy to complete the proof by combining over $k\ge 1$. 
\medskip

\noindent {\bf The KMT theorem for the SRW.}
As an application of Theorem \ref{coupling}, we give a new proof of Theorem \ref{mainthm} for the simple random walk. Although this is just a special case of the full theorem, it is important in its own right due to the importance of the SRW in various areas of science and mathematics. For instance, within the last ten years, the KMT embedding for the SRW played a pivotal role in the solution of a series of long-standing open questions about the simple random walk by the quartet of authors Dembo, Peres, Rosen, and Zeitouni \cite{dprz01, dprz04}. 

The proof of the KMT theorem for the SRW is obtained  using a combination of Theorem \ref{coupling}, Example 3, and an induction argument. The induction step involves proving the following theorem about sums of exchangeable binary variables. This seems to be a new result. 
\begin{thm}\label{mainthm2}
There exist positive universal constants $C$, $K$ and $\lambda_0$ such that the following is true. Take any integer $n\ge 2$. Suppose $\ve_1,\ldots,\ve_n$ are exchangeable $\pm 1$ random variables. For $k = 0,1,\ldots,n$, let $S_k = \sum_{i=1}^k \ve_i$ and let
$W_k = S_k - \frac{k}{n}S_n$.
It is possible to construct a version of $W_0,\ldots,W_n$ and a standard Brownian bridge $(\widetilde{B}_t)_{0\le t\le 1}$ on the same probability space such that for any $0 < \lambda < \lambda_0$, 
\[
\ee\exp(\lambda \max_{k\le n} |W_k - \sqrt{n}\widetilde{B}_{k/n}|) \le \exp(C\log n) \ee\exp\biggl(\frac{K\lambda^2 S_n^2}{n}\biggr).
\]
\end{thm}
Note that by Example 2, it is possible to use Theorem \ref{coupling} and induction whenever the summands have a density with respect to Lebesgue measure and the function $h$ is reasonably well-behaved. This holds, for instance, for log-concave densities, or densities of the type considered in Theorem \ref{auto}. In such cases it is not very difficult (although technically messier than the binary case) to prove a version of Theorem \ref{mainthm2} using the method of this paper. However, we do not know yet how to use Theorem~\ref{coupling} to prove the KMT theorem in its full generality, because we do not know how to generalize the smoothing technique of Example 3.

The theorem that we prove about the KMT coupling for the SRW, stated below, is somewhat stronger than existing results.
\begin{thm}\label{ourthm}
Let  $\ve_1,\ve_2,\ldots$ be i.i.d.\ symmetric $\pm1$-valued random variables. For each $k$, let $S_k := \sum_{i=1}^k \ve_i$. It is possible to construct a version of the sequence $(S_k)_{k\ge 0}$ and a standard Brownian motion $(B_t)_{t\ge 0}$  on the same probability space such that for all $n$ and all $x\ge 0$,
\[
\pp\bigl(\max_{k\le n} |S_k - B_k| \ge C\log n + x\bigr) \le K e^{-\lambda x},
\] 
where $C$, $K$, and $\lambda$ do not depend on $n$.
\end{thm}
The above result is stronger than the corresponding statement about the SRW implied by Theorem \ref{mainthm} because it gives a single coupling for the whole process, instead of giving different  couplings for different $n$. Such results have been recently established in the KMT theorem for summands with finite $p$th moment \cite{lifshits07, zaitsev08}.

The paper is organized as follows. In Section \ref{couplingtool}, we prove Theorem \ref{coupling}. Two versions of Example 3 are worked out in Section \ref{tusnady}. The main induction step, which proves Theorem \ref{mainthm2}, is carried out in Section \ref{induction}. Finally, the proof of Theorem~\ref{ourthm} is completed in Section \ref{complete}.


\section{Proof of Theorem \ref{coupling}}\label{couplingtool}
The proof will proceed as a sequence of lemmas. The lemmas will not be used in the subsequent sections, and only Theorem~\ref{coupling} is relevant for the future steps.

\begin{lmm}\label{schauder}
Let $n$ be a positive integer, and suppose $A$ is a continuous map from $\rr^n$ into the set of $n\times n$ positive semidefinite matrices. Suppose there exists a constant $b\ge 0$  such that for all $x\in \rr^n$,
\[
\|A(x)\|\le b.
\]
Then there exists a probability measure $\mu$ on $\rr^n$ such that if $X$ is a random vector following the law $\mu$, then
\begin{equation}\label{expbd}
\ee\exp\avg{\theta, X} \le \exp(b \|\theta\|^2)
\end{equation}
for all $\theta \in \rr^n$, and 
\begin{equation}\label{iparts2}
\ee\avg{X, \nabla f(X)} = \ee \tr(A(X) \hess f(X))
\end{equation}
for all $f\in C^2(\rr^n)$ such that the expectations $\ee|f(X)|^2$, $\ee\|\nabla f(X)\|^2$, and $\ee|\tr(A(X)\hess f(X))|$ are finite.
Here $\nabla f$ and $\hess f$ denote the gradient and Hessian of $f$, and $\tr$ stands for the trace of a matrix. 
\end{lmm}
\begin{proof}
Let $K$ denote the set of all probability measures $\mu$ on $\rr^n$ satisfying 
\[
\int x\mu(dx) = 0 \ \text{ and } \ \int \exp\smallavg{\theta, x} \mu(dx) \le \exp(b\|\theta\|^2) \ \text{for all $\theta\in\rr^n$}.
\]
It is easy to see by the Skorokhod representation theorem and Fatou's lemma that $K$ is a (nonempty) compact subset of the space $V$ of all finite signed measures on $\rr^n$ equipped with the topology of weak-* convergence (that is, the locally convex Hausdorff topology generated by the separating family of seminorms $|\mu|_f := |\int fd\mu|$, where $f$ ranges over all continuous functions with compact support). Also, obviously, $K$ is convex.

Now fix $\ve \in (0,1)$. Define a map $T_\ve : K \ra V$ as follows. Given $\mu\in K$, let $X$ and $Z$ be two independent random vectors, defined on some probability space, with $X\sim \mu$ and $Z$ following the standard gaussian law on $\rr^n$. Let $T_\ve\mu$ be the law of the random vector
\[
(1-\ve) X + \sqrt{2\ve A(X)} Z,
\]
where $\sqrt{A(X)}$ denotes the positive semidefinite square root of the matrix $A(X)$. Then for any $\theta \in\rr^n$,
\begin{align*}
\int \exp{\smallavg{\theta,x}} T_\ve\mu(dx) &= \ee\exp\avg{\theta, (1-\ve) X + \sqrt{2\ve A(X)} Z}\\
&= \ee\exp\bigl(\avg{\theta, (1-\ve) X} + \ve\avg{\theta, A(X)\theta}\bigr)\\
&\le \exp(b\ve\|\theta\|^2) \ee\exp\avg{\theta, (1-\ve) X}\\
&\le \exp(b\ve\|\theta\|^2 + b (1-\ve)^2 \|\theta\|^2).
\end{align*}
For $\ve\in (0,1)$, $1-\ve+\ve^2 \le 1$. Hence, $b\ve + b(1-\ve)^2 \le b$, and therefore $T_\ve$ maps $K$ into $K$. Since $A$ is a continuous map, and the transformation $A\mapsto \sqrt{A}$ is continuous (see e.g.\ \cite{bhatia97}, page 290, equation (X.2)), it is easy to see that $T_\ve$ is continuous under the weak-* topology. Hence, by the Schauder-Tychonoff fixed point theorem for locally convex topological vector spaces (see e.g.\ \cite{dunfordschwartz58}, Chapter V, 10.5), we see that $T_\ve$ must have a fixed point in  $K$. For each $\ve\in (0,1)$, let $\mu_\ve$ be a fixed point of $T_\ve$, and let $X_\ve$ denote a random vector following the law $\mu_\ve$. 

Now take any $f\in C^2(\rr^n)$ with $\nabla f$ and $\hess f$ bounded and uniformly continuous. Fix $\ve\in (0,1)$, and let
\[
Y_\ve = -\ve X_\ve + \sqrt{2\ve A(X_\ve)} Z.
\]
By the definition of $T_\ve \mu$, note that 
\begin{equation}\label{eqzero}
\ee\bigl(f(X_\ve + Y_\ve) - f(X_\ve)\bigr) = 0.
\end{equation}
Now let
\begin{align*}
\mathcal{R}_\ve &= f(X_\ve + Y_\ve) - f(X_\ve) - \avg{Y_\ve,  \nabla f(X_\ve)}  -   \frac{1}{2}\avg{Y_\ve, \hess f(X_\ve)\, Y_\ve}.
\end{align*}
First, note that
\begin{equation}\label{term1}
\ee\avg{Y_\ve, \nabla f(X_\ve)} = - \ve\ee\avg{X_\ve, \nabla f(X_\ve)}.
\end{equation}
By the definition of $K$, all moments of $\|X_\ve\|$ are bounded by constants that do not depend on $\ve$. Hence, as $\ve \ra 0$, we have
\begin{equation}\label{term2}
\begin{split}
\ee \avg{Y_\ve, \hess f(X_\ve)\, Y_\ve} &= 2\ve\ee\tr(\sqrt{A(X_\ve)} \hess f(X_\ve)\sqrt{A(X_\ve)}) + O(\ve^{3/2})\\
&= 2\ve\ee\tr(A(X_\ve) \hess f(X_\ve)) + O(\ve^{3/2}).
\end{split}
\end{equation}
Now, by the boundedness and uniform continuity of $\hess f$, one can see that
\[
|\mathcal{R}_\ve| \le \|Y_\ve\|^2 \delta(\|Y_\ve\|),
\]
where $\delta:[0,\infty)\ra [0,\infty)$ is a bounded function satisfying $\lim_{t\ra 0} \delta(t) = 0$. Now, by the nature of $K$, it is easy to verify that the moments of $\ve^{-1}\|Y_\ve\|^2$ can be bounded by constants that do not depend on $\ve$. Combining this with the above-mentioned properties of $\delta$ and the fact that $\|Y_\ve\|\ra 0$ in probability as $\ve\ra 0$, we get 
\begin{equation}\label{term3}
\lim_{\ve\ra 0} \ve^{-1}\ee|\mathcal{R}_\ve| = 0.
\end{equation}
Now let $\mu$ be a cluster point of the collection $\{\mu_\ve\}_{0< \ve < 1}$ as $\ve \ra 0$, and let $X$ denote a random variable following the law $\mu$. Such a cluster point exists because $K$ is a compact set. By uniform integrability, equations \eqref{eqzero}, \eqref{term1}, \eqref{term2}, \eqref{term3}, and the continuity of $A$, we get
\[
\ee\avg{X, \nabla f(X)} = \ee \tr(A(X) \hess f(X)).
\] 
This completes the proof for $f\in C^2(\rr^n)$ with $\nabla f$ and $\hess f$ bounded and uniformly continuous. Next, take any $f\in C^2(\rr^n)$. Let $g:\rr^n \ra [0,1]$ be a $C^\infty$ function such that $g(x) = 1$ if $\|x\|\le 1$ and $g(x)=0$ if $\|x\|\ge 2$. For each $a >1$, let $f_a(x) = f(x) g(a^{-1} x)$. Then $f_a\in C^2$ with $\nabla f_a$ and $\hess f_a$ bounded and uniformly continuous. Moreover, $f_a$ and its derivatives converge pointwise to those of $f$ as $a\ra \infty$, as is seen from the expressions
\begin{align*}
\fpar{f_a}{x_i} &= \fpar{f}{x_i}(x) g(a^{-1}x) + a^{-1} f(x) \fpar{g}{x_i}(a^{-1}x), \\
\mpar{f_a}{x_i}{x_j} &= \mpar{f}{x_i}{x_j}(x)g(a^{-1}x) + a^{-1}\fpar{f}{x_i}(x)\fpar{g}{x_j}(a^{-1}x) \\
&\qquad + a^{-1}\fpar{f}{x_j}(x)\fpar{g}{x_i}(a^{-1}x) + a^{-2}f(x) \mpar{g}{x_i}{x_j}(a^{-1}x).
\end{align*}
Since $\ee\|X\|^2< \infty$ and $\|A(x)\|\le b$, the above  expressions also show that if the expectations $\ee|f(X)|^2$, $\ee\|\nabla f(X)\|^2$, and $\ee|\tr(A(X) \hess f(X))|$ are finite, then we can apply the dominated convergence theorem to conclude that
\begin{align*}
&\lim_{a\ra \infty} \ee\avg{X,\nabla f_a(X)} = \ee\avg{X, \nabla f(X)} \ \text{and} \\
&\lim_{a\ra \infty}\ee\tr(A(X) \hess f_a(X)) = \ee\tr(A(X) \hess f(X)).
\end{align*}
This completes the proof.
\end{proof}

\begin{lmm}\label{devbd}
Let $A$ and $X$ be as in Lemma \ref{schauder}. Take any $1\le i< j\le n$. Let 
\[
v_{ij}(x) := a_{ii}(x)+a_{jj}(x)-2a_{ij}(x),
\]
where $a_{ij}$ denotes the $(i,j)$th element of $A$.
Then 
for all $\theta\in \rr$,
\[
\ee\exp(\theta|X_i-X_j|) \le 2\ee\exp(2\theta^2v_{ij}(X)).
\]
\end{lmm}
\begin{proof}
Take any positive integer $k$. Define $f:\rr^n \ra \rr$ as 
\[
f(x) := (x_i - x_j)^{2k}.
\]
Then a simple calculation shows that
\[
\avg{x, \nabla f(x)} = 2k(x_i - x_j)^{2k},
\]
and
\[
\tr(A(x) \hess f(x)) = 2k(2k-1) (x_i - x_j)^{2k-2} v_{ij}(x).
\]
The positive definiteness of $A$ shows that $v_{ij}$ is everywhere nonnegative. An application of H\"older's inequality now gives
\begin{align*}
\ee|\tr(A(X)\hess f(X))| \le 2k(2k-1) \bigl(\ee(X_i-X_j)^{2k}\bigr)^{\frac{k-1}{k}} \bigl(\ee v_{ij}(X)^k\bigr)^{\frac{1}{k}}.
\end{align*}
From the identity \eqref{iparts2} we can now conclude that
\[
\ee(X_i-X_j)^{2k} \le (2k-1) \bigl(\ee(X_i-X_j)^{2k}\bigr)^{\frac{k-1}{k}} \bigl(\ee v_{ij}(X)^k\bigr)^{\frac{1}{k}}.
\]
This shows that
\[
\ee(X_i - X_j)^{2k} \le (2k-1)^k \ee v_{ij}(X)^k.
\]
To complete the proof, note that
\begin{align*}
\ee\exp(\theta|X_i-X_j|) &\le 2\ee\cosh(\theta(X_i-X_j))\\
&= 2\sum_{k=0}^\infty \frac{\theta^{2k}\ee(X_i-X_j)^{2k}}{(2k)!}\\
&\le 2 + 2\sum_{k=1}^\infty \frac{(2k-1)^k \theta^{2k}\ee(v_{ij}(X)^k)}{(2k)!}.
\end{align*}
By the slightly crude but easy inequality
\[
\frac{(2k-1)^k}{(2k)!} \le \frac{2^k}{k!},
\]
the proof is done.
\end{proof}

\begin{lmm}\label{ipartslemma}
Suppose $\rho$ is a probability density function on $\rr$ which is positive on an interval (bounded or unbounded) and zero outside. Suppose $\int_{-\infty}^\infty x\rho(x)dx = 0$. For each $x$ in the support of $\rho$, let
\[
h(x) := \frac{\int_x^\infty y\rho(y) dy}{\rho(x)}.
\]
Outside the support, let $h \equiv 0$. 
Let $X$ be a random variable with density $\rho$ and finite second moment. Then 
\begin{equation}\label{iparts}
\ee(X\varphi(X)) = \ee(h(X)\vp(X))
\end{equation}
for each absolutely continuous $\varphi$ such that both sides are well defined and $\ee|h(X)\varphi(X)|< \infty$. Moreover, if $h_1$ is another function satisfying \eqref{iparts} for all Lipschitz $\varphi$,  then $h_1 = h$ a.e.\ on the support of $\rho$. 

Conversely, if $Y$ is a random variable such that \eqref{iparts} holds with $Y$ in place of $X$, for all $\varphi$ such that $|\varphi(x)|$, $|x\varphi(x)|$, and $|h(x)\vp(x)|$ are uniformly bounded, then $Y$ must have the density~$\rho$. 
\end{lmm}
\begin{proof}
Let $u(x) = h(x) \rho(x)$. Note that $u$ is continuous, positive on the support of $\rho$, and $\lim_{x\ra -\infty} u(x) = \lim_{x\ra \infty} u(x) = 0$ since
\[
u(x) = \int_x^\infty y \rho(y) dy = - \int_{-\infty}^x y\rho(y) dy.
\]
Note that the above identity holds because $\int_{-\infty}^\infty x\rho(x) dx = 0$. Again, by the assumption that $\ee(X^2) < \infty$, it is easy to verify that
\[
\ee( h(X)) = \int_{-\infty}^\infty u(x) dx = \ee(X^2) < \infty.
\]
When $\varphi$ is a bounded Lipschitz function, then \eqref{iparts} is just the integration by parts identity
\[
\int_{-\infty}^\infty x\varphi(x) \rho(x) dx = \int_{-\infty}^\infty \vp(x) u(x)dx.
\]
Now take any absolutely continuous $\varphi$ and a $C^\infty$ map $g:\rr \ra [0,1]$ such that $g(x) = 1$ on $[-1,1]$ and $g(x)=0$ outside $[-2,2]$. For each $a >1$, let
\[
\varphi_a(x) := \varphi(x) g(a^{-1}x).
\]
Then
\[
\vp_a(x) = \vp(x) g(a^{-1}x) + a^{-1}\varphi(x) g'(a^{-1}x).
\]
It is easy to see that $\varphi_a$ and $\vp_a$ are bounded, and they converge to $\varphi$ and $\vp$ pointwise as $a\ra \infty$. Moreover, $|x\varphi_a(x)|\le |x\varphi(x)|$ and 
\[
|h(x)\vp_a(x)| \le |h(x)\vp(x)| + a^{-1}\|g'\|_\infty|h(x)\varphi(x)|.
\] 
Since we have assumed that $\ee|X\varphi(X)|$, $\ee|h(X)\vp(X)|$, and $\ee|h(X)\varphi(X)|$ are finite, we can now apply the dominated convergence theorem to conclude that \eqref{iparts} holds for $\varphi$. 

Suppose $h_1$ is another function satisfying \eqref{iparts} for all Lipschitz $\varphi$ and $\ee(X^2)<\infty$. Let $\varphi(x)$ be a Lipschitz function such that $\vp(x) = \mathrm{sign}(h_1(x)-h(x))$. Then 
\[
0 = \ee(\vp(X)(h_1(X)-h(X))) = \ee|h_1(X)-h(X)|.
\]
This shows that $h_1 = h$ a.e.\ on the support of $\rho$.

For the converse, let $X$ have density $\rho$ and take any bounded continuous function $v:\rr\ra \rr$, let $m = \ee v(X)$, and define
\[
\varphi(x) := \frac{1}{u(x)}\int_{-\infty}^x \rho(y)(v(y) - m) dy = - \frac{1}{u(x)}\int_x^\infty \rho(y)(v(y) - m) dy
\]
on the support of $\rho$. 
Since $u$ is nonzero and absolutely continuous everywhere on the support of $\rho$, therefore $\varphi$  is well-defined and absolutely continuous. Next, we prove that $|x\varphi(x)|$ is uniformly bounded. If $x \ge 0$, then
\begin{align*}
|x\varphi(x)| &= \biggl|\frac{x}{u(x)}\int_x^\infty \rho(y)(v(y) - m) dy\biggr|\\
&\le \frac{2\|v\|_\infty}{|u(x)|}\int_x^\infty y \rho(y) dy = 2\|v\|_\infty.
\end{align*}
Similarly, the same bound holds for $x<0$. A direct verification shows that
\[
h(x) \vp(x) - x\varphi(x) = v(x) - m.
\]
Thus, $|h(x)\vp(x)|$ is uniformly bounded. Finally, by the continuity of $\varphi$, 
$|\varphi(x)| \le \sup_{|t|\le 1} |\varphi(t)| + |x\varphi(x)|$ is also uniformly bounded.

So, if $Y$ is a random variable such that \eqref{iparts} holds for $Y$ in place of $X$ and every $\varphi$ such that $|\varphi(x)|$, $|x\varphi(x)|$,  and $|h(x)\vp(x)|$ are uniformly bounded, then
\[
\ee v(Y) - \ee v(X) = \ee(v(Y) - m) = \ee(h(Y)\vp(Y)-Y\varphi(Y)) = 0.
\]
Thus, $Y$ must have the same distribution as $X$.
\end{proof}

\begin{proof}[Proof of Theorem \ref{coupling}]
First, assume $W$ has a density $\rho$ with respect to Lebesgue measure which is positive and continuous everywhere. Define $h$ in terms of $\rho$ as in the statement of Lemma \ref{ipartslemma}. Then by the second assertion of Lemma \ref{ipartslemma}, 
\[
h(w) = \ee(T | W=w) \ \text{a.s.}
\]
Note that $h$ is nonnegative by definition. So we can define a function $A$ from $\rr^2$ into the set of $2\times 2$ positive semidefinite matrices as 
\[
A(x_1,x_2) := \left(
\begin{array}{cc}
h(x_1) & \sigma\sqrt{h(x_1)}\\
\sigma\sqrt{h(x_1)} & \sigma^2
\end{array}
\right).
\]
Note that $A(x_1,x_2)$ does not depend on $x_2$ at all. It is easy to see that $A$ is positive semidefinite. Also, since $\rho$ is assumed to be continuous, therefore so are $h$ and $A$. Since $T$ is bounded by a constant, so is $h$. Let $X = (X_1, X_2)$ be a random vector satisfying \eqref{expbd} and~\eqref{iparts2} of Lemma~\ref{schauder} with this $A$. Take any absolutely continuous $\varphi:\rr \ra \rr$ such that $|\varphi(x)|$, $|x\varphi(x)|$, and $|h(x)\vp(x)|$ are uniformly bounded. Let $\Phi$ denote an antiderivative of $\varphi$, i.e.\ a function such  that $\Phi' = \varphi$. We can assume that $\Phi(0) = 0$. Define $f:\rr^2 \ra \rr$ as $f(x_1,x_2) := \Phi(x_1)$. Then for some constant $C$,  for all $x_1,x_2$, 
\begin{align*}
&|f(x_1,x_2)| \le C|x_1|, \ \|\nabla f(x_1,x_2)\| \le C, \\
&\text{ and } \ |\tr(A(x_1,x_2) \hess f (x_1,x_2))| \le C.
\end{align*}
Thus, we can apply Lemma \ref{schauder} to conclude that for this $f$,
\[
\ee\avg{X,\nabla f(X)} = \ee\tr(A(X) \hess f(X)),
\]
which can be written as
\[
\ee (X_1 \varphi(X_1)) = \ee(h(X_1) \vp(X_1)).
\]
Since this holds for all $\varphi$ such that $|\varphi(x)|$, $|x\varphi(x)|$, and $|h(x)\vp(x)|$ are uniformly bounded, Lemma \ref{ipartslemma} tells us that $X_1$ must have the same distribution as $W$. 

Similarly, taking any $\varphi$ such that $|\varphi(x)|$, $|x\varphi(x)|$, and $|\vp(x)|$ are uniformly bounded, letting $\Phi$ be an antiderivative of $\varphi$, and putting $f(x_1,x_2) = \Phi(x_2)$, we see that
\[
\ee(X_2 \varphi(X_2)) = \sigma^2 \ee(\vp(X_2)),
\] 
which implies that $X_2 \sim N(0,\sigma^2)$. We now wish to apply Lemma \ref{devbd} to the pair $(X_1,X_2)$. Note that 
\[
v_{12}(x_1,x_2) = h(x_1) + \sigma^2 - 2\sigma\sqrt{h(x_1)} = \bigl(\sqrt{h(x_1)} - \sigma\bigr)^2
\]
Since $h(x_1) \ge 0$, we have
\[
\bigl(\sqrt{h(x_1)} - \sigma\bigr)^2 = \frac{\bigl(h(x_1) - \sigma^2\bigr)^2}{\bigl(\sqrt{h(x_1)} + \sigma\bigr)^2} \le \frac{\bigl(h(x_1) - \sigma^2\bigr)^2}{\sigma^2}.
\]
Since $h(X_1)$ has the same distribution as $h(W)$, and $h(W) = \ee(T|W)$, the required bound can now be obtained using Lemma \ref{devbd} and Jensen's inequality.

So we have finished the proof when $W$ has a probability density $\rho$ with respect to Lebesgue measure which is positive and continuous everywhere.  Let us now drop that assumption, but keep all others. For each $\ve > 0$, let $W_\ve := W + \ve Y$, where $Y$ is an independent standard gaussian random variable. If $\nu$ denotes the law of $W$ on the real line, then $W_\ve$ has the probability density function
\[
\rho_\ve(x) = \int_{-\infty}^\infty \frac{e^{-(x-y)^2/2\ve^2}}{\sqrt{2\pi}\ve} d\nu(y).
\]
From the above representation, it is easy to deduce that $\rho_\ve$ is positive and continuous everywhere. Again, note that for any Lipschitz $\varphi$, 
\begin{align*}
\ee(W_\ve\varphi(W_\ve)) &= \ee(W \varphi(W+\ve Y)) + \ve \ee(Y\varphi(W + \ve Y))\\
&= \ee(T\vp(W+\ve Y)) + \ve^2 \ee(\vp(W+\ve Y)) \\
&= \ee((T+\ve^2)\vp(W_\ve)).
\end{align*}
(Note that in the second step, we required that \eqref{coeffdef} holds for {\it any} derivative of $\varphi$ instead of just one.) Thus, by what we have already proved, we can construct a version of $W_\ve$ and a $N(0,\sigma^2+ \ve^2)$ r.v.\ $Z_\ve$ on the same probability space such that for all~$\theta$,
\[
\ee\exp(\theta|W_\ve - Z_\ve|) \le 2\ee \exp\biggl(\frac{2\theta^2 (T-\sigma^2)^2}{\sigma^2 + \ve^2}\biggr). 
\]
Let $\mu_\ve$ be the law of the pair $(W_\ve,Z_\ve)$ on $\rr^2$. Clearly, $\{\mu_\ve\}_{\ve > 0}$ is a tight family. Let $\mu_0$ be a cluster point as $\ve \ra 0$, and let $(W_0,Z_0)\sim \mu_0$. Then $W_0$ has the same distribution as $W$, and $Z_0\sim N(0,\sigma^2)$. By the Skorokhod representation, Fatou's lemma, and the monotone convergence theorem, it is clear that
\[
\ee\exp(\theta|W_0 - Z_0|)\le \liminf_{\ve \ra 0} \ee\exp(\theta|W_\ve - Z_\ve|) \le 2\ee \exp\biggl(\frac{2\theta^2 (T-\sigma^2)^2}{\sigma^2}\biggr). 
\]
This completes the proof.
\end{proof}

\section{Elaborations on Example 3}
\label{tusnady}
The goal of this section is to prove the following two theorems. The first one is simply Example 3 from Section \ref{intro}. The second one can be called a conditional version of the same thing (which is harder to prove).

\begin{thm}\label{coupling3}
There exist universal constants $\kappa$ and $\theta_0 > 0$ such that the following is true. Let $n$ be a positive integer and let $\ve_1,\ldots,\ve_n$ be i.i.d.\ symmetric $\pm 1$ random variables. Let $S_n = \sum_{i=1}^n \ve_i$.  It is possible to construct a version of $S_n$ and $Z_n \sim N(0,n)$ 
on the same probability space such that 
\[
\ee \exp( \theta_0 |S_n - Z_n|) \le \kappa.
\]
\end{thm}
Note that by Markov's inequality, this implies exponentially decaying tails for $|S_n - Z_n|$, with a rate of decay that does not depend on $n$. 

\begin{thm}\label{coupling2}
Let $\ve_1,\ldots,\ve_n$ be $n$ arbitrary elements of $\{-1,1\}$. Let $\pi$ be a uniform random permutation of $\{1,\ldots,n\}$. For each $1\le k\le n$, let
$S_k = \sum_{\ell=1}^k \ve_{\pi(\ell)}$, 
and let 
\[
W_k = S_k - \frac{kS_n}{n}.
\]
There exist universal constants $c > 1$ and $\theta_0 >0$ satisfying the following. Take any $n\ge 3 $, any possible value of $S_n$, and any $n/3\le k\le 2n/3$. It is possible to construct a version of $W_k$ and a gaussian random variable $Z_k$ with mean~$0$ and variance $k(n-k)/n$  on the same probability space such that for any $\theta \le  \theta_0$, 
\[
\ee \exp( \theta |W_k - Z_k|) \le \exp\biggl( 1+ \frac{c\theta^2 S_n^2}{n}\biggr).
\]
\end{thm}
Both of the above theorems will be proved using Theorem \ref{coupling}. 
We proceed as before in a sequence of lemmas that are otherwise irrelevant for the rest of the manuscript (except Lemma \ref{moment2}, which has an important application later on).

\begin{lmm}\label{intparts}
Suppose $X$ and $Y$ are two independent random variables, with $X$ following the symmetric distribution on $\{-1,1\}$ and $Y$ following  the uniform distribution on $[-1,1]$. Then for any Lipschitz $\varphi$, we have
\[
\ee(X\varphi(X + Y)) = \ee( (1 - XY)\vp(X+Y)),
\]
and
\[
\ee(Y\varphi(X+Y)) = \frac{1}{2}\ee((1-Y^2)\vp(X+Y)).
\]
\end{lmm}
\begin{proof}
We have
\begin{align*}
\ee((1-XY)\vp(X+Y)) &= \frac{1}{4}\int_{-1}^1(1+y)\vp(-1+y) dy \\
&\qquad + \frac{1}{4}\int_{-1}^1 (1-y) \vp(1+y) dy.
\end{align*}
Integrating by parts, we see that
\begin{align*}
\int_{-1}^1(1+y)\vp(-1+y) dy &= 2\varphi(0) - \int_{-1}^1 \varphi(-1+y) dy,
\end{align*}
and
\begin{align*}
\int_{-1}^1(1-y)\vp(1+y) dy &= -2\varphi(0) + \int_{-1}^1 \varphi(1+y) dy.
\end{align*}
Adding up, we get
\begin{align*}
\ee((1-XY) \vp(X+Y)) &= \frac{1}{4} \int_{-1}^1 \varphi(1+y) dy - \frac{1}{4}\int_{-1}^1 \varphi(-1+y) dy\\
&= \ee(X\varphi(X+Y)).
\end{align*}
For the second part, just observe that for any $x$, integration by parts gives
\[
\frac{1}{2}\int_{-1}^1 y \varphi(x+y) dy = \frac{1}{2}\int_{-1}^1 \frac{1-y^2}{2}\vp(x+y) dy.
\]
This completes the proof.
\end{proof}

\begin{proof}[Proof of Theorem \ref{coupling3}]
For simplicity, let us write $S$ for $S_n$.
Let $Y$ be a random variable independent of $\ve_1,\ldots,\ve_n$ and uniformly distributed on the interval $[-1,1]$. Suppose we are given the values of $\ve_1, \ldots,\ve_{n-1}$. Let $\ee^-$ denote the conditional expectation given this information. Let 
\[
S^- = \sum_{i=1}^{n-1}\ve_i, \ \ X = \ve_n.
\]
Then Lemma \ref{intparts} gives
\begin{align*}
\ee^-(X\varphi(S^- + X + Y)) &= \ee^- ((1-XY)\vp(S + Y))\\
&= \ee^- ((1-\ve_n Y)\vp(S + Y)).
\end{align*}
Taking expectation on both sides we get
\[
\ee(\ve_n\varphi(S + Y)) = \ee((1-\ve_n Y)\vp(S + Y)).
\]
By symmetry, this gives
\[
\ee(S\varphi(S+Y)) = \ee((n - SY) \vp(S+Y)).
\]
Again, by Lemma \ref{intparts}, we have
\[
\ee(Y \varphi(S+Y)) = \frac{1}{2}\ee((1-Y^2)\vp(S+Y)).
\]
Thus, putting $\ts = S+Y$ and 
\[
T = n - SY + \frac{1-Y^2}{2},
\]
we have
\begin{equation}\label{tseq}
\ee(\ts \varphi(\ts)) = \ee(T\vp(\ts)).
\end{equation}
Let $\sigma^2 = n$. Then 
\begin{align*}
\frac{(T-\sigma^2)^2}{\sigma^2} &\le \frac{2S^2+ \frac{1}{2}}{n}.
\end{align*}
Now, clearly, $\ee(\ts) = 0$ and $\ee(\ts^2) <\infty$. The equation \eqref{tseq} holds and the random variable $T$ is a.s.\  bounded.  Therefore, all conditions for applying  Theorem \ref{coupling} to $\ts$ are met, and hence we can conclude that it is possible to construct a version of $\ts$ and a $N(0,\sigma^2)$ random variable $Z$ on the same space such that for all~$\theta$,
\[
\ee\exp(\theta|\ts-Z|) \le 2\ee\exp(2\theta^2\sigma^{-2}(T-\sigma^2)^2).
\]
Since the value of $S$ is determined if we know $\ts$, we can now construct a version of $S$ on the same probability space satisfying $|S - \ts|\le 1$. It follows that
\[
\ee\exp(\theta|S-Z|) \le 2\ee\exp(|\theta| + 2\theta^2\sigma^{-2}(T-\sigma^2)^2).
\]
Using the bound on $(T-\sigma^2)^2/\sigma^2$ obtained above, we have
\[
\ee\exp(\theta|S-Z|) \le 2\exp(|\theta| + \theta^2/n) \ee \exp(4\theta^2 S^2/n).
\]
To complete the argument, note that if $V$ is a standard gaussian r.v., independent of $S$, then 
\begin{align*}
\ee \exp(4\theta^2 S^2/n) &= \ee\exp(\sqrt{8}\theta VS/\sqrt{n})\\
&= \ee(\ee(\exp(\sqrt{8}\theta V\ve_1/\sqrt{n})|V)^n)\\
&= \ee(\cosh^n(\sqrt{8}\theta V/\sqrt{n})).
\end{align*}
Using the simple inequality $\cosh x \le \exp x^2$, this gives 
\begin{equation}\label{s2bd}
\ee \exp(4\theta^2 S^2/n) \le \ee\exp(8\theta^2 V^2)  = \frac{1}{\sqrt{1-16\theta^2}} \ \text{ if } 16\theta^2 < 1.
\end{equation}
The conclusion now follows by choosing $\theta_0$ sufficiently small. 
\end{proof}

\begin{lmm}\label{momentbd}
Let all notation be as in the statement of Theorem \ref{coupling2}. 
Then for any $\theta \in \rr$ and any $1\le k\le n$, we have
\[
\ee\exp(\theta W_k/\sqrt{k}) \le \exp \theta^2.
\]
\end{lmm}
\noindent {\it Remark.} Note that the bound does not depend on the value of $S_n$. This is crucial for the next lemma and the induction step later on. Heuristically, this phenomenon is not mysterious because the centered process $(W_k)_{k\le n}$ has maximum freedom to fluctuate when $S_n = 0$.
\begin{proof}
Fix $k$, and let $m(\theta) := \ee\exp(\theta W_k/\sqrt{k})$. Since $W_k$ is a bounded random variable, there is no problem in showing that $m$ is differentiable and 
\[
m'(\theta) = \frac{1}{\sqrt{k}}\ee(W_k \exp(\theta W_k/\sqrt{k})). 
\]
Now note that
\begin{align*}
\frac{1}{n}\sum_{i=1}^k \sum_{j=k+1}^n (\ve_{\pi(i)} - \ve_{\pi(j)}) &= \frac{(n-k) \sum_{i=1}^k \ve_{\pi(i)} - k \sum_{j=k+1}^n \ve_{\pi(j)}}{n}\\
&= \frac{(n-k) \sum_{i=1}^k \ve_{\pi(i)} - k (S_n - \sum_{i=1}^k \ve_{\pi(i)})}{n}\\
&= \sum_{i=1}^k \ve_{\pi(i)} - \frac{kS_n}{n} = W_k.
\end{align*}
Thus, 
\begin{equation}\label{mprime}
m'(\theta) = \frac{1}{n\sqrt{k}}\sum_{i=1}^k \sum_{j=k+1}^n \ee((\ve_{\pi(i)} - \ve_{\pi(j)})\exp(\theta W_k /\sqrt{k})).
\end{equation}
Now fix $i \le k < j$. Let $\pi' = \pi\circ (i,j)$, so that $\pi'(i) = \pi(j)$ and $\pi'(j) = \pi(i)$. Then $\pi'$ is again uniformly distributed on the set of all permutations of $\{1,\ldots,n\}$. Moreover, $(\pi,\pi')$ is an {\it exchangeable pair} of random variables. Let 
\[
W_k' = \sum_{\ell=1}^k \ve_{\pi'(\ell)} - \frac{k S_n}{n}.
\]
Then 
\begin{align*}
\ee((\ve_{\pi(i)} - \ve_{\pi(j)})\exp(\theta W_k/ \sqrt{k})) &= \ee((\ve_{\pi'(i)} - \ve_{\pi'(j)})\exp(\theta W_k' /\sqrt{k}))\\
&= \ee((\ve_{\pi(j)} - \ve_{\pi(i)})\exp(\theta W_k' /\sqrt{k})).
\end{align*}
Averaging the two equal quantities, we get
\begin{align*}
&\ee((\ve_{\pi(i)} - \ve_{\pi(j)})\exp(\theta W_k/ \sqrt{k}))\\
&= \frac{1}{2} \ee((\ve_{\pi(i)} - \ve_{\pi(j)})(\exp(\theta W_k/ \sqrt{k}) - \exp(\theta W_k'/\sqrt{k}))).
\end{align*}
Thus, from the inequality
\[
|e^x - e^y| \le \frac{1}{2}|x-y|(e^x + e^y)
\]
and the fact that $W_k - W_k' = \ve_{\pi(i)} - \ve_{\pi(j)}$, we get
\begin{align*}
&\bigl|\ee((\ve_{\pi(i)} - \ve_{\pi(j)})\exp(\theta W_k/ \sqrt{k}))\bigr|\\
&\le \frac{|\theta|}{4\sqrt{k}} \ee((\ve_{\pi(i)} - \ve_{\pi(j)})^2(\exp(\theta W_k/ \sqrt{k}) + \exp(\theta W_k'/\sqrt{k})))\\
&\le \frac{|\theta|}{\sqrt{k}} \ee(\exp(\theta W_k/ \sqrt{k}) + \exp(\theta W_k'/\sqrt{k}))\\
&= \frac{2|\theta|}{\sqrt{k}} \ee\exp(\theta W_k/ \sqrt{k}) = \frac{2|\theta|}{\sqrt{k}} m(\theta).
\end{align*}
Using this estimate in \eqref{mprime}, we get
\[
|m'(\theta)|\le\frac{2|\theta|}{nk}\sum_{i=1}^k \sum_{j=k+1}^n m(\theta) \le 2|\theta|m(\theta).
\]
Using that $m(0) =1$, it is now easy to complete the proof.
\end{proof}

\begin{lmm}\label{moment2}
Let us continue with the notation of Theorem \ref{coupling2}. There exists a universal constant $\alpha_0 >0$ such that for all $n$, all possible values of $S_n$, all $k$ such that $k\le 2n/3$, and all $\alpha \le \alpha_0$, we have
\[
\ee \exp(\alpha S_k^2/k) \le \exp\biggl(1 + \frac{3\alpha S_n^2}{4n}\biggr).
\]
\end{lmm}
\noindent {\it Remark.} The exact value of the constant $3/4$ in the above bound is not important; what is important is that the constant is $ < 1$ as long as we take $k\le 2n/3$. This is why the induction argument can be carried out in Section \ref{induction}. However, there is no mystery; the fact that one can always get a constant $< 1$ can be explained via simple heuristic arguments once Lemma \ref{momentbd} is known.
\begin{proof}
Let $Z$ be an independent standard gaussian random variable. Then
\begin{align*}
\ee\exp(\alpha S_k^2/k) &= \ee\exp\biggl(\sqrt{\frac{2\alpha}{k}} Z S_k\biggr)\\
&= \ee\exp\biggl(\sqrt{\frac{2\alpha}{k}} Z W_k  + \sqrt{\frac{2\alpha}{k}}\frac{kS_n}{n} Z \biggr).
\end{align*}
Now, by Lemma \ref{momentbd} we have
\[
\ee\biggl(\exp\biggl(\sqrt{\frac{2\alpha}{k}} Z W_k \biggr)\biggl| Z\biggr) \le \exp(2\alpha Z^2).
\]
Thus, we have
\begin{align*}
\ee\exp(\alpha S_k^2/k) &\le \ee\exp\biggl(2\alpha Z^2  + \sqrt{\frac{2\alpha}{k}}\frac{kS_n}{n} Z \biggr).
\end{align*}
Since $S_n$ is nonrandom, the right hand side is just the expectation of a function of a standard gaussian random variable, which can be easily computed. This gives, for $0< \alpha < 1/4$,
\[
\ee\exp(\alpha S_k^2/k) \le \frac{1}{\sqrt{1-4\alpha}} \exp\biggl(\frac{\alpha k S_n^2}{(1-4\alpha)n^2}\biggr).
\]
The lemma is now proved by bounding $k$ by $2n/3$ and choosing $\alpha_0$ small enough to ensure that $1/(1-4\alpha_0)$ is sufficiently close to $1$. 
\end{proof}

\begin{proof}[Proof of Theorem \ref{coupling2}]
For simplicity, we shall write $W$ for $W_k$ and $S$ for $S_n$, but $S_k$ will be written as usual. 

Let $Y$ be a random variable independent of $\pi$ and uniformly distributed on the interval $[-1,1]$. Fix $1\le i\le k$ and $k< j\le n$. Suppose we are given the values of $\{\pi(\ell), \ell\ne i,j\}$. Let $\ee^-$ denote the conditional expectation given this information. Let 
\[
S^- = \sum_{\ell\ne i,j} \ve_{\pi(\ell)}, \ \ W^- = \sum_{\ell\le k, \ell \ne i} \ve_{\pi(\ell)} - \frac{kS}{n}.
\]
If $S\ne S^-$, then we must have $\ve_{\pi(i)} = \ve_{\pi(j)}$, and hence in that case 
\[
\ee^-((\ve_{\pi(i)}-\ve_{\pi(j)}) \varphi(W + Y)) = 0.
\]
Next let us consider the only other possible scenario, $S = S^-$. Then the conditional distribution of $\ve_{\pi(i)}-\ve_{\pi(j)}$ is symmetric over $\{-2,2\}$. Let
\[
X = \frac{\ve_{\pi(i)}-\ve_{\pi(j)}}{2} = \ve_{\pi(i)},
\]
and note that 
\[
W = W^- + X.
\]
Thus, under $S=S^-$, Lemma \ref{intparts} shows that for all Lipschitz $\varphi$,
\begin{align*}
\ee^-((\ve_{\pi(i)}-\ve_{\pi(j)}) \varphi(W + Y)) &= 2\ee^-(X\varphi(W^- + X + Y)) \\
&= 2\ee^- ((1-XY)\vp(W + Y))\\
&= \ee^-((2-(\ve_{\pi(i)}-\ve_{\pi(j)})Y) \vp(W+Y)).
\end{align*}
Next, let
\[
a_{ij} := 1-\ve_{\pi(i)}\ve_{\pi(j)} - (\ve_{\pi(i)}-\ve_{\pi(j)}) Y. 
\]
A simple verification shows that
\[
a_{ij} = 
\begin{cases}
2-(\ve_{\pi(i)}-\ve_{\pi(j)})Y &\text{ if } \ve_{\pi(i)} \ne \ve_{\pi(j)} \\
0 &\text{ if } \ve_{\pi(i)} = \ve_{\pi(j)}.
\end{cases}
\]
Thus, irrespective of whether $S = S^-$ or $S\ne S^-$, we have
\[
\ee^-((\ve_{\pi(i)}-\ve_{\pi(j)}) \varphi(W + Y)) = \ee^-(a_{ij}\vp(W+Y)).
\]
Clearly, we can now replace $\ee^-$ by $\ee$ in the above expression. Now, as in the proof of Lemma \ref{momentbd}, observe that
\begin{align*}
W = \frac{1}{n}\sum_{i=1}^k \sum_{j=k+1}^n (\ve_{\pi(i)} - \ve_{\pi(j)}).
\end{align*}
Combining the last two observations, we have
\[
\ee(W\varphi(W+Y)) = \ee\biggl(\biggl(\frac{1}{n}\sum_{i=1}^k \sum_{j=k+1}^n a_{ij}\biggr) \vp(W+Y)\biggr).
\]
Again, by Lemma \ref{intparts}, we have
\[
\ee(Y \varphi(W+Y)) = \frac{1}{2}\ee((1-Y^2)\vp(W+Y)).
\]
Thus, putting $\tw = W+Y$ and 
\[
T = \frac{1}{n}\sum_{i=1}^k \sum_{j=k+1}^n a_{ij} + \frac{1-Y^2}{2},
\]
we have
\begin{equation}\label{tweq}
\ee(\tw \varphi(\tw)) = \ee(T\vp(\tw)).
\end{equation}
Now
\begin{align*}
\frac{1}{n}\sum_{i=1}^k \sum_{j=k+1}^n a_{ij} &= \frac{k(n-k)}{n} - \frac{(\sum_{i=1}^k \ve_{\pi(i)})(\sum_{j=k+1}^n \ve_{\pi(j)})}{n} - WY.
\end{align*}
Let $\sigma^2 = k(n-k)/n$. Since $n/3\le k\le 2n/3$ and $|W| \le |S_k| + \frac{2}{3}|S|$, a simple computation gives 
\begin{align*}
\frac{(T-\sigma^2)^2}{\sigma^2} &\le \frac{n}{k(n-k)} (|S_k| + |W| + 1/2)^2\\
&\le C\biggl(\frac{S_k^2}{k} + \frac{S^2}{n} + 1\biggr),
\end{align*}
where $C$ is a universal constant.

Now, clearly, $\ee(\tw) = 0$ and $\ee(\tw^2)<\infty$. The equation \eqref{tweq} holds and the random variable $T$ is a.s.\ bounded.  Therefore, all conditions for applying Theorem \ref{coupling} to $\tw$ are met, and hence we can conclude that it is possible to construct a version of $\tw$ and a $N(0,\sigma^2)$ random variable $Z$ on the same space such that for all $\theta$,
\[
\ee\exp(\theta|\tw-Z|) \le 2\ee\exp(2\theta^2\sigma^{-2}(T-\sigma^2)^2).
\]
Since the value of $W$ is determined if we know $\tw$, we can now construct a version of $W$ on the same probability space satisfying $|W - \tw|\le 1$. It follows that
\[
\ee\exp(\theta|W-Z|) \le 2\ee\exp(|\theta| + 2\theta^2\sigma^{-2}(T-\sigma^2)^2).
\]
Using the bound on $(T-\sigma^2)^2/\sigma^2$ obtained above, we have
\[
\ee\exp(\theta|W-Z|) \le 2\exp(|\theta| + C\theta^2 S^2/n + C\theta^2) \ee \exp(C\theta^2 S_k^2/k),
\]
where, again, $C$ is a universal constant.
The conclusion now follows from Lemma \ref{moment2} by choosing $\theta$ sufficiently small. 
\end{proof}

\section{The induction step}\label{induction}
The goal of this section is to prove the following theorem, which couples a pinned random walk with a Brownian Bridge. The tools used are Theorem~\ref{coupling2} and induction. The induction hypothesis, properly formulated, allows us to get rid of the dyadic construction of the usual KMT proofs. The following is an alternative statement of Theorem \ref{mainthm2}, given here for the convenience of the reader.
\begin{thm}\label{maincoupling}
Let us continue with the notation of Theorem \ref{coupling2}. There exist positive universal constants $C$, $K$ and $\lambda_0$ such that the following is true. For any $n\ge 2$, and any possible value of $S_n$, it is possible to construct a version of $W_0,W_1,\ldots,W_n$ and gaussian r.v.\ $Z_0,Z_1,\ldots,Z_n$ with mean zero and
\begin{equation}\label{covform}
\cov(Z_i,Z_j) = \frac{(i\wedge j)(n- (i\vee j))}{n}
\end{equation}
on the same probability space such that for any $\lambda \in (0,\lambda_0)$, 
\[
\ee\exp(\lambda \max_{i\le n} |W_i - Z_i|) \le \exp\biggl(C\log n + \frac{K\lambda^2 S_n^2}{n}\biggr).
\]
\end{thm}
\begin{proof}
Recall the universal constants $\alpha_0$ from Lemma \ref{moment2} and $c$ and $\theta_0$ from Theorem \ref{coupling2}. We contend that for carrying out the induction step, it suffices to take
\begin{align}\label{consts}
K = 8 c, \  \lambda_0 \le  \sqrt{\frac{\alpha_0}{16c}} \wedge \frac{\theta_0}{2}, \ \text{and} \ C \ge \frac{1 + \log 2}{\log (3/2)}.
\end{align}
Choosing the constants to satisfy these constraints, we will now prove the claim by induction on $n$. Now, for each $n$, and each possible value $a$ of $S_n$, let $f^n_a(\bs)$ denote the discrete probability density function of the sequence $(S_0,S_1, \ldots, S_n)$. Note that this is just the uniform distribution over $\ma^n_a$, where 
\begin{equation}\label{adef}
\ma_a^n := \{\bs \in \zz^{n+1}:  s_0 = 0, \ s_n = a, \ \text{and} \ |s_i-s_{i-1}| = 1 \ \text{for all} \ i.\}
\end{equation}
Thus, for any $\bs\in \ma_a^n$, 
\begin{equation}\label{fform}
f_a^n(\bs) = \frac{1}{|\ma_a^n|}.
\end{equation}
Let $\phi^n(\bz)$ denote the probability density function of a gaussian random vector $(Z_0,\ldots,Z_n)$ with mean zero and covariance \eqref{covform}.

We want to show that for each $n$, and each possible value $a$ of $S_n$, we can construct a joint probability density $\rho_a^n(\bs,\bz)$ on $\zz^{n+1}\times \rr^{n+1}$ such that 
\begin{equation}\label{marg1}
\int \rho_a^n (\bs,\bz)\, d\bz = f_a^n (\bs), \ \ \int  \rho_a^n (\bs, \bz) \, d\bs = \phi^n(\bz),
\end{equation}
and for each $\lambda < \lambda_0$, 
\[
\int \exp\biggl(\lambda \max_{i\le n} \biggl| s_i - \frac{ia}{n} - z_i\biggr| \biggr) \rho^n_a (\bs,\bz) \,d\bs \, d\bz \le \exp\biggl(C\log n + \frac{K\lambda^2 a^2}{n}\biggr).
\]
Suppose $\rho_a^k$ can be constructed for $k = 1,\ldots,n-1$, for allowed values of $a$ in each case. We will now demonstrate a construction of $\rho_a^n$ when $a$ is an allowed value for $S_n$. 

First, fix a possible value $a$ of $S_n$ and an index $k$ such that $n/3\le k\le 2n/3$ (for definiteness, take $k = [n/2]$). Given $S_n = a$, let $g^{n,k}_a(s)$ denote the density function of $S_k$. Recall the definition \eqref{adef} of $\ma_a^n$ and note that for all allowed values of $s$ of $S_k$, an elementary counting argument gives
\begin{equation}\label{gform}
g^{n,k}_a(s) = \frac{|\ma_s^k| |\ma_{a-s}^{n-k}|}{|\ma_a^n|}.
\end{equation}
Let $h^{n,k}(z)$ denote the density function of the gaussian distribution with mean $0$ and variance $k(n-k)/n$. By Theorem \ref{coupling2} and the inequality $\exp|x|\le \exp(x) + \exp(-x)$, we see that there exists a joint density function $\psi^{n,k}_a(s,z)$ on $\zz\times \rr$ such that 
\begin{equation}\label{marg2}
\int \psi^{n,k}_a(s,z) \, dz = g^{n,k}_a(s), \ \ \int \psi^{n,k}_a(s,z) \, ds = h^{n,k}(z),
\end{equation}
and for all $0 < \theta \le \theta_0$, 
\begin{equation}\label{tterm}
\int\exp\biggl(\theta\biggl|s - \frac{ka}{n} - z\biggr|\biggr) \psi^{n,k}_a(s,z)\, ds\, dz\le \exp\biggl(1 + \frac{c\theta^2 a^2}{n}\biggr).
\end{equation}
Now define a function $\gamma_a^n: \zz\times\rr \times \zz^{k+1}\times \rr^{k+1}\times \zz^{n-k+1}\times \rr^{n-k+1} \ra \rr$ as follows:
\begin{equation}\label{gammadef}
\gamma_a^n (s,z,\bs,\bz,\bs',\bz') := \psi_a^{n,k}(s,z) \rho_s^k(\bs,\bz) \rho_{a-s}^{n-k}(\bs',\bz').
\end{equation}
By integrating over $\bs', \bz'$, then $\bs,\bz$, and finally $s,z$, it is easy to verify that $\gamma_a^n$ is a probability density function (if either $a$ or $s$ is not an allowed value, then $\psi_a^{n,k}(s,z) = 0$, so there is no problem). 

Let $(S,Z, \bbs, \bbz, \bbs', \bbz')$ denote a random vector following the density $\gamma_a^n$. In words, this means the following: We are first generating $(S,Z)$ from the joint distribution $\psi_a^{n,k}$; given $S=s, Z=z$, we are independently generating the pairs $(\bbs, \bbz)$ and $(\bbs', \bbz')$ from the joint densities $\rho_s^k$ and $\rho_{a-s}^{n-k}$ respectively. 

Now define two random vectors $\bby\in \rr^{n+1}$ and $\bbu\in \zz^{n+1}$ as follows. For $i\le k$, let
\[
Y_i = Z_i + \frac{i}{k}Z,
\]
and for $i\ge k$, let
\[
Y_i = Z'_{i-k} + \frac{n-i}{n-k}Z.
\]
Note that the two definitions match at $i=k$ because $Z_k = Z'_0 = 0$. Next, define $U_i = S_i$ for $i\le k$ and $U_i = S + S_{i-k}'$ for $i\ge k$. Again, the definitions match at $i=k$ because $S_k = S$ and $S'_0 = 0$. {\it We claim that the joint density of $(\bbu, \bby)$ is a valid candidate for $\rho_a^n$.} The claim is proved in several steps.
\vskip.1in
1. {\it Marginal distribution of $\bbu$.} From equations \eqref{marg1} and \eqref{marg2} it is easy to see that
\[
\int \gamma_a^n(s,z,\bs,\bz,\bs',\bz') \,  d\bz\, d\bz'\, dz = g_a^{n,k}(s) f_s^k(\bs)f_{a-s}^{n-k}(\bs').
\]
In other words, the distribution of the triplet $(S,\bbs, \bbs')$ can be described as follows: Generate $S$ from the distribution of $S_k$ given $S_n = a$; then independently generate $\bbs$ and $\bbs'$ from the conditional distributions $f_s^k$ and $f_{a-s}^{n-k}$. It should now be intuitively clear that $\bbu$ has marginal density $f_a^n$. Still, to be completely formal, we apply equations \eqref{fform} and \eqref{gform} to get
\[
g_a^{n,k}(s) f_s^k(\bs)f_{a-s}^{n-k}(\bs') = \frac{|\ma_s^k||\ma_{a-s}^{n-k}|}{|\ma_a^n|} \frac{1}{|\ma_s^k|} \frac{1}{|\ma_{a-s}^{n-k}|} = \frac{1}{|\ma_a^n|},
\]
and observe that there is a one-to-one correspondence between $(S,\bbs,\bbs')$ and $\bbu$, and $\bbu$ can take any value in $\ma_a^n$. 
\vskip.1in
2. {\it Marginal distribution of $\bby$.} First, we claim that $Z$, $\bbz$, and $\bbz'$ are independent with densities $h^{n,k}$, $\phi^k$, and $\phi^n$ respectively. Again, using \eqref{marg1} and \eqref{marg2}, this is easily seen as follows.
\begin{align*}
\int \gamma^n_a(s,z,\bs,\bz,\bs',\bz') \, d\bs' \, d\bs\, ds &= \int \psi_a^{n,k}(s,z) \rho_s^k(\bs,\bz) \rho_{a-s}^{n-k}(\bs',\bz')\, d\bs' \, d\bs\, ds\\
&=  \phi^{n-k}(\bz')\int \psi_a^n(s,z) \rho_s^k(\bs,\bz) d\bs\, ds\\
&= \phi^{n-k}(\bz')\phi^k(\bz) \int \psi_a^n(s,z) ds\\
&= \phi^{n-k}(\bz')\phi^k(\bz)h^{n,k}(z).
\end{align*}
Thus, $\bby$ is a gaussian random vector with mean zero. It only remains to compute $\cov(Y_i,Y_j)$. Considering separately the cases $i\le j\le k$, $k\le i\le j$, and $i\le k\le j$, it is now  straightforward to verify that $\cov(Y_i,Y_j) = i(n-j)/n$ in each case. Thus, $\bby \sim \phi^n$. 
\vskip.1in
3. {\it The exponential bound.} For $0\le i\le n$, let
\[
W_i = U_i - \frac{ia}{n}.
\]
We have to show that for $0< \lambda <\lambda_0$, 
\[
\ee\exp(\lambda \max_{i\le n} |W_i - Y_i|)\le \exp\biggl(C \log n + \frac{K\lambda^2 a}{n}\biggr),
\]
where $C$, $K$, and $\lambda_0$ are as in \eqref{consts}.
Now let
\[
T_L := \max_{i\le k} \biggl|S_i - \frac{iS}{k} - Z_i\biggr|, \ \ T_R := \max_{i\ge k} \biggl|S'_{i-k} - \frac{i-k}{n-k}(a-S) - Z'_{i-k}\biggr|,
\]
and
\[
T := \biggl|S - \frac{ka}{n} - Z\biggr|. 
\]
We claim that
\begin{equation}\label{maxineq}
\max_{i\le n} |W_i - Y_i|\le \max\{T_L, T_R\} + T.
\end{equation}
To prove this, first take any $i\le k$. Then
\begin{align*}
|W_i - Y_i| &= \biggl|S_i - \frac{ia}{n} - \biggl(Z_i + \frac{iZ}{k}\biggr)\biggr|\\
&\le \biggl|S_i - \frac{iS}{k} - Z_i\biggr| + \biggl|\frac{iS}{k} - \frac{ia}{n} -  \frac{iZ}{k}\biggr|\\
&\le T_L + \frac{i}{k} T \le T_L + T.
\end{align*}
Similarly, for $i\ge k$,
\begin{align*}
|W_i - Y_i| &= \biggl|S + S_{i-k}' - \frac{ia}{n} - \biggl(Z'_{i-k} + \frac{n-i}{n-k}Z\biggr)\biggr|\\
&\le \biggl|S_{i-k}' - \frac{i-k}{n-k}(a-S) - Z'_{i-k}\biggr|\\
&\qquad + \biggl|S + \frac{i-k}{n-k}(a-S) - \frac{ia}{n} - \frac{n-i}{n-k}Z\biggr|\\
&= \biggl|S_{i-k}' - \frac{i-k}{n-k}(a-S) - Z'_{i-k}\biggr| + \frac{n-i}{n-k}\biggl|S - \frac{ka}{n} - Z\biggr|\\
&\le T_R + T.
\end{align*}
This proves \eqref{maxineq}. Now fix $\lambda < \lambda_0$. Using the crude bound $\exp(x\vee y) \le \exp x + \exp y$, we get
\begin{equation}\label{maxbd}
\exp(\lambda \max_{i\le n}|W_i - Y_i|) \le \exp(\lambda T_L + \lambda T) + \exp(\lambda T_R + \lambda T).
\end{equation}
Now, by the construction \eqref{gammadef}, it is easy to check that given $(S,Z) = (s,z)$, the conditional density of $(\bbs,\bbz)$ is simply $\rho_s^k$. By the induction hypothesis, this implies that
\[
\ee(\exp(\lambda T_L)| S,Z) \le \exp\biggl(C \log k + \frac{K \lambda^2 S^2}{k}\biggr).
\]
It is easy to see that the moment generating functions of both $T_L$ and $T$ are finite everywhere, and hence there is no problem in applying the Cauchy-Schwarz inequality to get
\begin{align*}
\ee\exp(\lambda T_L + \lambda T)&\le \bigl[\ee\bigl(\ee(\exp(\lambda T_L)| S,Z)^2\bigr) \ee(\exp(2\lambda T))\bigr]^{1/2}\\
&\le \exp(C \log k) \biggl[\ee\exp\biggl(\frac{2K\lambda^2S^2}{k}\biggr) \ee \exp(2\lambda T)\biggr]^{1/2}.
\end{align*}
We wish to apply Lemma \ref{moment2} to bound the first term inside the bracket. Observe that by \eqref{consts}, we have
\[
2K\lambda^2 \le 16 c\cdot \frac{\alpha_0}{16c} = \alpha_0,
\]
and also $n/3\le k\le 2n/3$ by assumption. Hence Lemma \ref{moment2} can indeed be applied to get
\[
\ee\exp\biggl(\frac{2K\lambda^2S^2}{k}\biggr) \le \exp\biggl(1 + \frac{3K\lambda^2 a^2}{2n}\biggr).
\]
Next, note that by \eqref{consts}, $2\lambda \le \theta_0$. Hence by inequality \eqref{tterm} with $\theta = 2\lambda$, we get the bound
\[
\ee\exp(2\lambda T) \le \exp\biggl(1 + \frac{4c\lambda^2a^2}{n}\biggr).
\]
Combining the last three steps, we have
\begin{align*}
\ee\exp(\lambda T_L + \lambda T)&\le \exp\biggl(C\log k + 1 + \frac{(3K + 8c)\lambda^2a^2}{4n}\biggr).
\end{align*}
Now, by \eqref{consts}, $3K + 8c = 4K$. Again, since $n/3\le k\le 2n/3$, we have
\[
\log k = \log n - \log (n/k)\le \log n - \log (3/2).
\]
Thus,
\begin{align*}
\ee\exp(\lambda T_L + \lambda T)&\le 2^{1/2}\exp\biggl(C\log n - C\log (3/2) + 1 + \frac{K\lambda^2a^2}{n}\biggr).
\end{align*}
By the symmetry of the situation, we can get the exact same bound on $\ee\exp(\lambda T_R + \lambda T)$. Combined with \eqref{maxbd}, this gives
\begin{align*}
\ee\exp(\lambda \max_{i\le n}|W_i - Y_i|)&\le 2\exp\biggl(C\log n - C\log (3/2) + 1 + \frac{K\lambda^2a^2}{n}\biggr).
\end{align*}
Finally, from the condition on $C$ in \eqref{consts}, we see that 
\[
-C \log (3/2) + 1 + \log 2\le 0.
\]
This completes the induction step. To complete the argument, we just choose $C$ so large and $\lambda_0$ so small that the result is true for $n = 2$ even if the vectors $(W_0,W_1,W_2)$ and $(Z_0,Z_1, Z_2)$ are chosen to be independent of each other.
\end{proof}

\section{Completing the proofs of the main theorems}\label{complete}
In this final section, we put together the pieces to complete the proofs of Theorem \ref{mainthm2} and Theorem \ref{ourthm}. The following lemma combines Theorem \ref{maincoupling} and Theorem \ref{coupling3} to  give a `finite $n$ version' of Theorem~\ref{ourthm}.
\begin{lmm}\label{mainlmm}
There exist universal constants $B> 1$ and $\lambda > 0$ such that the following is true.
Let $n$ be a positive integer and let  $\ve_1,\ve_2,\ldots, \ve_n$ be i.i.d.\ symmetric $\pm 1$ random variables. Let $S_k = \sum_{i=1}^k \ve_i$, $k=0,1,\ldots,n$. It is possible to construct a version of the sequence $(S_k)_{k\le n}$ and gaussian random variables $(Z_k)_{k\le n}$ with mean~$0$ and $\cov(Z_i,Z_j) = i\wedge j$  on the same probability space such that $\ee\exp(\lambda |S_n - Z_n|) \le B$ and 
\[
\ee \exp( \lambda\max_{k\le n} |S_k - Z_k|) \le B\exp(B\log n).
\]
\end{lmm}
\begin{proof}
Recall the universal constants $\theta_0$ and $\kappa$ from Theorem \ref{coupling3} and $C$, $K$, and $\lambda_0$ from Theorem \ref{maincoupling}. Choose $\lambda$ so small that 
\[
\lambda < \frac{\theta_0\wedge \lambda_0}{2} \ \text{ and } \ 16K\lambda^2 < 1.
\]
Let the probability densities $f_a^n$, $\rho_a^n$, and $\phi^n$ be as in the proof of Theorem~\ref{maincoupling}. Let $g^n$ and $h^n$ denote the densities of $S_n$ and $Z_n$ respectively. By Theorem~\ref{coupling3} and the choice of $\lambda$, there is a joint density $\psi^n$ on $\zz\times \rr$ such that
\[
\int \psi^n(s,z) \, dz = g^n(s), \ \ \int\psi^n(s,z) \, ds = h^n(z),
\]
and 
\begin{equation}\label{szeq}
\int \exp(2\lambda |s - z|) \psi^n(s,z) \, ds \, dz \le \kappa.
\end{equation}
Now define a function $\gamma^n: \zz \times \rr \times \zz^{n+1}\times \rr^{n+1} \ra \rr$ as
\[
\gamma^n(s,z,\bs,\bz) := \psi^n(s,z)\rho_s^n(\bs,\bz).
\]
It is easy to check that this is a probability density function. Let $(S,Z,\bbs,\bbz)$ be a random vector following this density. As in the proof of Theorem \ref{maincoupling}, an easy integration shows that the joint density of $(Z,\bbz)$ is simply
\[
h^n(z) \phi^n(\bz).
\]
Define a random vector $\bby = (Y_0,\ldots,Y_n)$ as 
\[
Y_i = Z_i + \frac{i}{n}Z.
\]
By the independence of $Z$ and $\bbz$ and their distributions, it follows that $\bby$ is a mean zero gaussian random vector with $\cov(Y_i, Y_j) = i\wedge j$. 

Next, integrating out $z$ and $\bz$ we see that the joint density of $(S,\bbs)$ is
\[
g^n(s)f^n_s(\bs).
\]
Elementary probabilistic reasoning now shows that the marginal distribution of $\bbs$ is the same as that of a simple random walk up to time $n$. 

Let us now show that the law of  the pair $(\bbs, \bby)$ satisfies the conditions of the theorem.
First, let $W_i = S_i - iS/n$. Note that for any $i\le n$,
\begin{align*}
|S_i - Y_i| &= \biggl|S_i - \biggl(Z_i + \frac{i}{n}Z\biggr)\biggr|\\
&\le |W_i - Z_i| + \frac{i}{n}|S - Z|.
\end{align*}
Note that the conditional distribution of $(\bbs, \bbz)$ given $(S,Z) = (s,z)$ is simply $\rho_s^n$. Since $\lambda < \lambda_0$, we have by the construction of $\rho_s^n$ that
\[
\ee\bigl(\exp(\lambda \max_{i\le n} |W_i - Z_i|)\bigl| S, Z\bigr) \le \exp\biggl( C\log n + \frac{K\lambda^2S^2}{n}\biggr).
\]
Thus, using the Cauchy-Schwarz inequality and \eqref{szeq}, we can now get
\begin{align*}
&\ee\exp(\lambda\max_{i\le n}|S_i - Y_i|)\\
&\le \bigl[\ee\bigl(\ee\bigl(\exp(\lambda \max_{i\le n} |W_i - Z_i|)\bigl| S, Z\bigr)^2\bigr) \ee\exp(2\lambda |S-Z|)\bigr]^{1/2}\\
&\le \exp(C\log n) \bigl[\kappa\ee\exp(2K\lambda^2 S^2/n) \bigr]^{1/2}.
\end{align*}
By inequality \eqref{s2bd} and the choice of $\lambda$, the proof of the maximal inequality is done. For the other inequality, note that we have \eqref{szeq} and $Y_n = Z$ since $Z_n = 0$.
\end{proof}

\begin{proof}[Proofs of Theorems \ref{mainthm2} and \ref{ourthm}]
The proof of Theorem \ref{mainthm2} follows trivially from Theorem \ref{maincoupling}. The proof of Theorem \ref{ourthm} also follows quite easily from Lemma \ref{mainlmm}, but some more work is required. We carry out the few remaining steps below.

For $r =1,2,\ldots$ let $m_r = 2^{2^r}$, and $n_r = m_r - m_{r-1}$. For each $r$ $(S^{(r)}_k, Z^{(r)}_k)_{0\le k\le n_r}$ be a random vector satisfying the conclusions of Lemma \ref{mainlmm}, and suppose these random vectors are independent. Inductively define an infinite sequence $(S_k,Z_k)_{k\ge 0}$ as follows. Let $S_k = S^{(1)}_k$ and $Z_k = Z^{(1)}_k$ for $k\le m_1$. Having defined $(S_k,Z_k)_{k\le m_{r-1}}$, define $(S_k,Z_k)_{m_{r-1} < k\le m_r}$ as
\[
S_k := S_{k-m_{r-1}}^{(r)} + S_{m_{r-1}}, \ \ Z_k := Z_{k-m_{r-1}}^{(r)} + Z_{m_{r-1}}.
\]
Clearly, since the increments are independent, $S_k$ and $Z_k$ are indeed random walks with binary and gaussian increments respectively.

Now recall the constants $B$ and $\lambda$ in Lemma \ref{mainlmm}. 
First, note that for each $r$, by Lemma \ref{mainlmm} and independence we have 
\begin{equation}\label{firstineq}
\begin{split}
\ee\exp(\lambda|S_{m_r}-Z_{m_r}|) &\le \ee \exp\biggl(\lambda \sum_{\ell=1}^r |S_{n_\ell}^{(\ell)} - Z_{n_{\ell}}^{(\ell)}|\biggr)\\
&= \prod_{\ell=1}^r \ee \exp\bigl(\lambda |S_{n_\ell}^{(\ell)} - Z_{n_{\ell}}^{(\ell)}|\bigr)\le B^r.
\end{split}
\end{equation}
Next, let 
\[
C = \frac{1}{1 - \frac{\exp(-\frac{1}{2}B \log 4)}{B}}.
\]
We will show by induction that for each $r$,
\begin{equation}\label{inducstep}
\ee\exp(\lambda\max_{k\le m_r} |S_k - Z_k|) \le CB^r \exp(B\log m_r).
\end{equation}
By Lemma \ref{mainlmm} and the facts that $B > 1$ and $C>1$, this holds for $r = 1$. Suppose it holds for $r-1$. By the inequality $\exp(x\vee y) \le \exp x + \exp y$, we have
\begin{equation}\label{induc}
\begin{split}
\ee\exp(\lambda\max_{k\le m_r} |S_k - Z_k|) &\le \ee\exp(\lambda\max_{m_{r-1}\le k\le m_r} |S_k - Z_k|) \\
&\qquad + \ee\exp(\lambda\max_{k\le m_{r-1}} |S_k - Z_k|).
\end{split}
\end{equation}
Let us consider the first term. We have
\[
\max_{m_{r-1}\le k\le m_r} |S_k - Z_k| \le \max_{1\le j\le n_r} |S_j^{(r)} - Z_j^{(r)}| + |S_{m_{r-1}} - Z_{m_{r-1}}|.
\]
Thus, by independence and Lemma \ref{mainlmm}, and the inequality \eqref{firstineq}, we get
\[
\ee\exp(\lambda\max_{m_{r-1}\le k\le m_r} |S_k - Z_k|) \le B^r\exp(B\log m_r).
\]
By the induction hypothesis and the relation $m_r = m_{r-1}^2$, we see that the second term in \eqref{induc} has the bound
\begin{align*}
\ee\exp(\lambda\max_{k\le m_{r-1}} |S_k - Z_k|) &\le CB^{r-1} \exp(B \log m_{r-1}) \\
&= CB^{r-1} \exp\biggl(\frac{B \log m_r}{2}\biggr).
\end{align*}
Combining, we get
\[
\ee\exp(\lambda\max_{k\le m_r} |S_k - Z_k|) \le B^r\exp(B\log m_r) \biggl(1+ \frac{C}{B}\exp\biggl(-\frac{B\log m_r}{2}\biggr)\biggr).
\]
From the definition of $C$, it easy to verify (since $m_r \ge 4$), that the term within the parentheses in the above expression is bounded by $C$. This completes the induction step.

So we have now shown \eqref{inducstep}. Since $r \le const. \log m_r$, this shows that there exists a constant $K$ such that for all $r$,
\[
\ee\exp(\lambda \max_{k\le m_r}|S_k - Z_k|)\le K \exp(K\log m_r).
\]
Now let us prove such an inequality for arbitrary $n$ instead of $m_r$. Take any $n\ge 2$. Let $r$ be such that $m_{r-1} \le n \le m_r$. Then $m_r = m_{r-1}^2 \le n^2$. Thus,
\begin{align*}
\ee\exp(\lambda \max_{k\le n}|S_k - Z_k|) &\le \ee\exp(\lambda \max_{k\le m_r}|S_k - Z_k|)\\
&\le K\exp(K\log m_r) \le K\exp(2K \log n).
\end{align*}
It is now easy to complete the argument using Markov's inequality.
\end{proof}
\vskip.1in
\noindent{\bf Acknowledgments.} The author is particularly indebted to David Mason and Andrei Zaitsev for clearing up many misconceptions about the literature and providing very helpful guidance. The author thanks Mikl\'os Cs\"org\H{o}, Persi Diaconis, Yuval Peres, Peter Bickel, Craig Evans, and Raghu Varadhan for useful  discussions and advice; and Ron Peled, Arnab Sen, Partha Dey, and Shankar Bhamidi for comments about the manuscript. Special thanks are due to Partha Dey for a careful verification of the proofs.

\end{document}